\documentclass[a4paper,twoside,10pt,reqno]{article}
\usepackage[english]{babel}
\usepackage{graphicx}
\usepackage{epstopdf}
\usepackage[T1]{fontenc}
\usepackage[latin1]{inputenc}
\usepackage{amssymb,amsmath,amsthm,dsfont,mathrsfs}
\usepackage{amsfonts,euscript}
\usepackage{color}
\usepackage{authblk}

\usepackage{authblk}
\usepackage{fullpage}
\usepackage{setspace}
\newtheorem{defi}{Definition}[section]
\newtheorem{teo}{Theorem}[section]
\newtheorem{pro}[teo]{Proposition}

\newtheorem{lem}[teo]{Lemma}

\newtheorem{ex}{Example}[section]

\newcommand{\er}{\mathbb{R}}

\newcommand{\dst}{\displaystyle}

\newcommand{\nn}{\nonumber}
\newcommand{\nid}{\noindent }
\newcommand{\mb}{\mathbf }
\newcommand{\ea}{\emph{et al.}}
\newcommand{\es}{\left}
\newcommand{\di}{\right}

\begin{document}

\title{Multidimensional extremal dependence coefficients}
\author{Helena Ferreira}
\affil{Universidade da Beira Interior, Centro de Matem\'{a}tica e Aplica\c{c}\~oes (CMA-UBI), Avenida Marqu\^es d'Avila e Bolama, 6200-001 Covilh\~a, Portugal\\ \texttt{helena.ferreira@ubi.pt}}

\author{Marta Ferreira}
\affil{Center of Mathematics of Minho University\\ Center for Computational and Stochastic Mathematics of University of Lisbon \\
Center of Statistics and Applications of University of Lisbon, Portugal\\ \texttt{msferreira@math.uminho.pt} }

\date{}

\maketitle

\abstract{Extreme values modeling has attracting the attention of researchers in diverse areas such as the environment, engineering, or finance. Multivariate extreme value distributions are particularly suitable to model the tails of multidimensional phenomena. The analysis of the dependence among multivariate maxima is useful to evaluate risk. Here we present new multivariate extreme value models, as well as, coefficients to assess multivariate extremal dependence.}\\

\nid\textbf{keywords:} {multivariate extreme value models, tail dependence, extremal coefficients, random fields}\\

\nid\textbf{AMS 2000 Subject Classification}: 60G70\\

\section{Introduction}\label{sint}
Let $\mb{X}=\{X(\mb{x}),\mb{x}\in\er^m\}$ be a random field. For a fixed set of locations $L=\{\mb{x}_1,\hdots,\, \,\mb{x}_d\}\subset\er^m$ and some partition $L_1=\{\mb{x}_1,\hdots,\, \mb{x}_{i_1}\}$,
$L_2=\{\mb{x}_{i_1+1},\hdots,\, \mb{x}_{i_2}\}$, \ldots, $L_p=\{\mb{x}_{i_{p-1}+1},\hdots,\, \mb{x}_{d}\}$, with $1\leq p\leq d$, consider the random vectors $\mb{X}_{L_1}=(X(\mb{x}_1),\hdots,\, X(\mb{x}_{i_1}))$, \ldots, $\mb{X}_{L_p}=(X(\mb{x}_{i_{p-1}+1}),\hdots,\, X(\mb{x}_{d}))$. We are going to evaluate the dependence between the vectors through coefficients, that is, the dependence between the marginals of $\mb{X}$ over disjoint regions $L_1,\hdots,\,  L_p$. Examples of applications within this context can be found in Naveau \emph{et al.} (\cite{nav+09} 2009) and Guillou \emph{et al.} (\cite{gui+14} 2014) for $d=p=2$, i.e.,  two locations, in Fonseca \emph{et al.} (\cite{fon+15} 2015) for $d>2$ and $p=2$, i.e., two group of several locations and Ferreira and Pereira (\cite{fer+per15} 2015) for $d=p>2$, i.e., several isolated locations.

In the applications, in order to study the dependence between sub-vectors of $\mb{X}$ we can form an auxiliary vector $(Y_1, \hdots, Y_p)$ where each variable $Y_i$ somehow summarizes the information of $\mb{X}_{L_i}$, $i = 1, \hdots, p$, and study the dependence between the variables $Y_i$. This is the approach followed by some authors (Naveau \emph{et al.} \cite{nav+09} 2009; Marcon \ea~\cite{marc+16} 2016).  In our proposal to infer the dependence between clusters of variables, we deal directly with the vectors $\mb{X}_{L_i}$, $i = 1, \hdots, p$. On the other hand, if the random field is vectorial, that is, for each location $\mb{x}_i$, $X(\mb{x}_i)$ is a vector $(X ^ 1 (\mb{x}_i), \hdots, X ^ s (\mb{x}_i))$, whenever we think of the dependence between $X(\mb{x}_1)$, \ldots, $X(\mb{x}_d)$ we have dependency between vectors.

The dependence between the random vectors $\mb{X}_{L_1}$, $\mb{X}_{L_2}$, \ldots, $\mb{X}_{L_p}$ can be characterized through the exponent measure
\begin{eqnarray}\nn
\ell_{\mb{x}_1,\hdots,\, \mb{x}_d}(t_1,\hdots,\, t_d)=-\ln  F_{(X(\mb{x}_1),\hdots,\, X(\mb{x}_{d}))}(t_1,\hdots,\, t_d),
\end{eqnarray}
where $F_{(X(\mb{x}_1),\hdots,\, X(\mb{x}_{d}))}$ denotes the distribution function (df) of $\mb{X}_{L}=(X(\mb{x}_1),\hdots,\, X(\mb{x}_{d}))$. If $\mb{X}$ is a max-stable random field with unit Fr\'echet marginals, then $\ell_{\mb{x}_1,\hdots,\, \mb{x}_d}$ is homogeneous of order $-1$ and the polar transformation used in the Pickands representation allows us to see it as a moment-based tail dependence tool (see, e.g., Finkenst\"{a}dt and Rootzén \cite{fin+roo03} 2003 or Beirlant \ea~\cite{beirl+04} 2004).\\

Our proposal also addresses $\ell_{\mb{x}_1,\hdots,\, \mb{x}_d}$ as a function of moments of transformations of $\mb{X}_{L}$. Specifically, the moments
\begin{eqnarray}\nn
e(\lambda_1,\hdots,\, \lambda_p)=E\left(\bigvee_{j=1}^p\bigvee_{x_i\in L_j}F_{X(\mb{x}_i)}^{\lambda_j}(X(\mb{x}_i))\right),\,(\lambda_1,\hdots,\, \lambda_p)\in(0,\infty)^p\,,
\end{eqnarray}
where $a\vee b=\max(a,b)$. If $p=d=2$, $\frac{1}{2}e(\lambda,1-\lambda)$ equals the $\lambda$-madogram of Naveau \emph{et al.} (\cite{nav+09} 2009), unless the addition of constant $\frac{1}{2}(E(U^{\lambda})+E(U^{1-\lambda}))$ where $U$ is standard uniform. When  $p=d\geq 2$, $e(\lambda_1^{-1},\hdots,\, \lambda_d^{-1})$ with $\sum_{j=1}^d\lambda_j=1$ equals the generalized madogram considered in Marcon \ea~(\cite{marc+16} 2016), unless the addition of constant $\frac{1}{d}\sum_{j=1}^dE\left(U^{\lambda_j^{-1}}\right)$.

Here we  also consider a shifted $e(\lambda_1,\hdots,\, \lambda_p)$ by subtracting the constant
$$
\frac{1}{p}\sum_{i=1}^pE\left(\bigvee_{x_i\in L_j}F_{X(\mb{x}_i)}^{\lambda_j}(X(\mb{x}_i))\right).
$$

The referred works consider max-stable random fields with standard Fr\'echet marginals, except Guillou \ea~(\cite{gui+14} 2014) where $\ell_{x_1,x_2}(t_1,t_2)$ is homogeneous of order $-1/\eta$ and $F_{X(x_i)}(t)=P(X(x_i)\leq t)=\exp(-\sigma(x_i)t^{-1/\eta})$, $i=1,2$, $\eta\in(0,1]$, corresponding to the bivariate extreme values model obtained in Ramos and Ledford (\cite{ram+led11} 2011). \\

We will also consider that $F_{(X(\mb{x}_1),\hdots,\, X(\mb{x}_{d}))}$ is such that $\ell_{\mb{x}_1,\hdots,\, \mb{x}_d}(t_1,\hdots,\, t_d)$ is homogeneous of order $-1/\eta$ and $F_{X(x)}(t)=P(X(x)\leq t)=\exp(-\sigma(x)t^{-1/\eta})$ for some constants $\sigma(x)>0$ and $\eta\in(0,1]$. Under this hypothesis, which includes all the other mentioned works whenever $\eta=1$ and $\sigma(x)=1$, we define extremal dependence functions that provide us coefficients to measure the dependence among $\mb{X}_{L_1}$, \ldots, $\mb{X}_{L_p}$ through the dependence between $M(L_j)$, $j=1,\hdots,\, p$
and
relate the extremal coefficients with the upper tail dependence function introduced in Ferreira and Ferreira (\cite{fer+fer12b} 2012) (Section \ref{scoef}). We  compute the extremal coefficients for several choices of $F_{(X(\mb{x}_1),\hdots,\, X(\mb{x}_{d}))}$ in Section \ref{sex}. Finally we consider an asymptotic tail independence coefficient to measure an ``almost" independence for a class of models wider than max-stable ones (Section \ref{sati}).

In order to simplify notations, we will write $X_i$ instead of $X(\mb{x}_i)$ and, for any vector $\mb{a}$ and any subset of its indexes $S$, we will write $\mb{a}_S$ to denote the sub-vector of $\mb{a}$ with indexes in $S$.

\section{Model and coefficients of multivariate extremal dependence}\label{scoef}

Let $I=\{1,\hdots,\, d\}$ and $I_1=\{\alpha(I_1)=1,\hdots,\, \omega(I_1)\}$, $I_2=\{\alpha(I_2)=\omega(I_1)+1,\hdots,\, \omega(I_2)\}$, \ldots, $I_p=\{\alpha(I_p)=\omega(I_{p-1})+1,\hdots,\, \omega(I_p)=d\}$ be a partition of $I$, $1\leq p\leq d$. Consider $\mb{X}_I=(X_1,\hdots,\, X_d)$ has df $F_{_{\mb{X}_I}}$ and univariate marginals $F_{i}$ such that
\begin{enumerate}
\item[(i)] $F_{i}(t)=\exp\left(-\sigma_i t^{-1/\eta}\right)$, $i=1,\hdots,\, d$
\item[(ii)] $\ell_{_{\mb{X}_I}}(t_1,\hdots,\, t_d)=-\ln  F_{_{\mb{X}_I}}(t_1,\hdots,\, t_d)$ is homogeneous of order $-1/\eta$,
\end{enumerate}
for some constants $\sigma_i>0$ and $\eta\in(0,1]$. Thus, the copula $C_{_{\mb{X}_I}}$ of $F_{_{\mb{X}_I}}$ is max-stable, i.e.
\begin{eqnarray}
C_{_{\mb{X}_I}}(u_1^s,\hdots,\, u_d^s)=C_{_{\mb{X}_I}}^s(u_1,\hdots,\, u_d),\,s>0.
\end{eqnarray}

In the following we use notation $M(I)=\bigvee_{i\in I}F_i(X_i)$.

\begin{lem}\label{lem1}
If $\mb{X}_I=(X_1,\hdots,\, X_d)$ satisfies conditions (i) and (ii) then, for all $(u_1,\hdots,\, u_p)\in (0,1)^p$,
\begin{eqnarray}\nn
P(M(I_1)\leq u_1,\hdots,\, M(I_p)\leq u_p)=
\exp\left\{-\ell_{_{\mb{X}_I}}\left(\sum_{j=1}^p\left(-\frac{\sigma_1}{\ln u_j}\right)^{\eta}\delta_1(I_j),\hdots,\, \sum_{j=1}^p\left(-\frac{\sigma_d}{\ln u_j}\right)^{\eta}\delta_d(I_j)\right)\right\}.
\end{eqnarray}

\end{lem}

\begin{proof}
We have successively
\begin{eqnarray}\nn
\begin{array}{rl}
&\dst P(M(I_1)\leq u_1,\hdots,\, M(I_p)\leq u_p)\\\\
=&\dst C_{_{\mb{X}_I}}\left(\sum_{j=1}^pu_j\delta_1(I_j),\hdots,\, \sum_{j=1}^pu_j\delta_d(I_j)\right)\\\\
=& \dst
\exp\left\{-\ell_{_{\mb{X}_I}}
\left(F_{1}^{-1}\left(\sum_{j=1}^pu_j\delta_1(I_j)\right),\hdots,\,
F_{d}^{-1}\left(\sum_{j=1}^pu_j\delta_d(I_j)\right)\right)\right\}.
\end{array}
\end{eqnarray}
\end{proof}

Analogously, we obtain, for $1\leq j<j'\leq p$,
\begin{eqnarray}\nn
\begin{array}{rl}
&\dst P(M(I_j)\leq u_j,M(I_{j'})\leq u_{j'})\\\\
=& \dst
\exp\left\{-\ell_{_{\mb{X}_{I_j\cup I_{j'}}}}\left(\sum_{i\in\{j,j'\}}\left(-\frac{\sigma_{\alpha(I_j\cup I_{j'})}}{\ln u_i}\right)^{\eta}\delta_{\alpha(I_j\cup I_{j'})}(I_i),\hdots,\, \sum_{i\in\{j,j'\}}\left(-\frac{\sigma_{\omega(I_j\cup I_{j'})}}{\ln u_i}\right)^{\eta}\delta_{\omega(I_j\cup I_{j'})}(I_i)\right)\right\},
\end{array}
\end{eqnarray}
where $\alpha(I_j\cup I_{j'})$ and $\omega(I_j\cup I_{j'})$ denote the first and last point of $I_j\cup I_{j'}$, respectively.

\begin{lem}\label{lem2}
If $\mb{X}_I=(X_1,\hdots,\, X_d)$ satisfies conditions (i) and (ii) then, for all $(\lambda_1,\hdots,\, \lambda_p)\in (0,\infty)^p$,
\begin{eqnarray}\label{lem2.1}
\dst E\left(\bigvee_{j=1}^p M(I_j)^{\lambda_j}\right)=
\frac{\dst \ell_{_{\mb{X}_I}}\left(\sigma_1^{\eta}\sum_{j=1}^p\lambda_j^{\eta}\delta_1(I_j),\hdots,\, \,
\sigma_d^{\eta}\sum_{j=1}^p\lambda_j^{\eta}\delta_d(I_j)\right)}
{\dst 1+\ell_{_{\mb{X}_I}}\left(\sigma_1^{\eta}\sum_{j=1}^p\lambda_j^{\eta}\delta_1(I_j),\hdots,\, \,
\sigma_d^{\eta}\sum_{j=1}^p\lambda_j^{\eta}\delta_d(I_j)\right)}.
\end{eqnarray}

\end{lem}

\begin{proof}
From Lemma \ref{lem1} and by applying the homogeneity of order $-1/\eta$ of $\ell_{\mb{X}_I}$, we have
\begin{eqnarray}\nn
\dst P\left(\bigvee_{j=1}^p M(I_j)\leq u^{\lambda_j^{-1}}\right)
=u^{\ell_{_{\mb{X}_I}}\left(\sigma_1^{\eta}\sum_{j=1}^p\lambda_j^{\eta}\delta_1(I_j),\hdots,\,
\sigma_d^{\eta}\sum_{j=1}^p\lambda_j^{\eta}\delta_d(I_j)\right)}
\end{eqnarray}
and
\begin{eqnarray}\nn
\dst E\left(\bigvee_{j=1}^p M(I_j)^{\lambda_j}\right)=\dst\int_0^1 u^{\ell_{_{\mb{X}_I}}\left(\sigma_1^{\eta}\sum_{j=1}^p\lambda_j^{\eta}\delta_1(I_j),\hdots,\,
\sigma_d^{\eta}\sum_{j=1}^p\lambda_j^{\eta}\delta_d(I_j)\right)}\ell_{_{\mb{X}_I}}\left(\sigma_1^{\eta}\sum_{j=1}^p\lambda_j^{\eta}\delta_1(I_j),\hdots,\,
\sigma_d^{\eta}\sum_{j=1}^p\lambda_j^{\eta}\delta_d(I_j)\right) du,
\end{eqnarray}
which leads to the result.
\end{proof}

The natural extension of the madogram to our context is the function
$$
\nu_{_{\mb{X}_{I_1},\hdots,\, \mb{X}_{I_p}}}(\lambda_1,\hdots,\,\lambda_p)=e(\lambda_1,\hdots,\, \lambda_p)-\frac{1}{p}\sum_{i=1}^pE\left(M(I_j)^{\lambda_j}\right),\,(\lambda_1,\hdots,\, \lambda_p)\in(0,\infty)^p.
$$

Motivated by the relation between $E\left(\bigvee_{j=1}^p M(I_j)^{\lambda_j}\right)$ and $\ell_{_{\mb{X}_I}}$ presented in Lemma \ref{lem2}, we first propose the following definition for the extremal dependence function between $\mb{X}_{I_1},\hdots,\,\mb{X}_{I_p}$.

\begin{defi}\label{def1}
If $\mb{X}_I=(X_1,\hdots,\, X_d)$ satisfies conditions (i) and (ii) then the extremal dependence function $\varepsilon_{_{\mb{X}_{I_1},\hdots,\, \mb{X}_{I_p}}}(\lambda_1,\hdots,\,\lambda_p)$ among $\mb{X}_{I_1},\hdots,\, \mb{X}_{I_p}$ is defined by
\begin{eqnarray}\nn
\dst \varepsilon_{_{\mb{X}_{I_1},\hdots,\, \mb{X}_{I_p}}}(\lambda_1,\hdots,\,\lambda_p)=\frac{ E\left(\bigvee_{j=1}^p M(I_j)^{\lambda_j}\right)}{1-E\left(\bigvee_{j=1}^p M(I_j)^{\lambda_j}\right)},\, (\lambda_1,\hdots,\, \lambda_p)\in (0,\infty)^p\,.
\end{eqnarray}
\end{defi}

As a consequence of Lema \ref{lem2} and Definition \ref{def1} which compares the distances of $E\left(\bigvee_{j=1}^p M(I_j)^{\lambda_j}\right)\in(0,1)$ to zero and one, we have the following property that discloses $\varepsilon_{_{\mb{X}_{I_1},\hdots,\, \mb{X}_{I_p}}}(\lambda_1,\hdots,\,\lambda_p)$ as a measure of the dependence between ${{\mb{X}_{I_1},\hdots,\, \mb{X}_{I_p}}}$.

\begin{pro}\label{pro1}
If $\mb{X}_I=(X_1,\hdots,\, X_d)$ satisfies conditions (i) and (ii) then, for all $(\lambda_1,\hdots,\, \lambda_p)\in (0,\infty)^p$,
\begin{eqnarray}\nn
\dst \varepsilon_{_{\mb{X}_{I_1},\hdots,\, \mb{X}_{I_p}}}(\lambda_1,\hdots,\,\lambda_p)=
\ell _{_{\mb{X}_I}}\left(\sigma_1^{\eta}\sum_{j=1}^p\lambda_j^{\eta}\delta_1(I_j),\hdots,\,
\sigma_d^{\eta}\sum_{j=1}^p\lambda_j^{\eta}\delta_d(I_j)\right).
\end{eqnarray}
\end{pro}

Therefore, the extremal dependence function among $\mb{X}_{I_1},\hdots,\, \mb{X}_{I_p}$ at the point $(\lambda_1,\hdots,\,\lambda_p)$ coincides with the tail dependence function of $\mb{X}_I$ at the point
$$((\sigma_1\lambda_1)^{\eta},\hdots,\,(\sigma_{\omega(I_1)}\lambda_1)^{\eta},
\,(\sigma_{\alpha(I_2)}\lambda_2)^{\eta},\hdots,\,(\sigma_{\omega(I_2)}\lambda_2)^{\eta},
\hdots,\,(\sigma_{\alpha(I_p)}\lambda_p)^{\eta},\hdots,
\,(\sigma_{\omega(I_p)}\lambda_p)^{\eta}).
$$
In the context of the validity of conditions (i) and (ii), by Proposition \ref{pro1}, we have
\begin{eqnarray}\label{ce1}
\dst \varepsilon_{_{\mb{X}_{I_1},\hdots,\, \mb{X}_{I_p}}}(1,\hdots,1)=
\ell _{_{\mb{X}_I}}\left(\sigma_1^{\eta},\hdots,\,
\sigma_d^{\eta}\right),
\end{eqnarray}
\begin{eqnarray}\nn
\dst {\varepsilon_{_{\mb{X}_{I_j}, \mb{X}_{I_{j'}}}}}(1,1)=
{\ell_{_{\mb{X}_{I_j\cup I_{j'}}}}}\left(\sigma_{\alpha(I_j)}^{\eta},\hdots,\,
\sigma_{\omega(I_j)}^{\eta},\,\sigma_{\alpha(I_{j'})}^{\eta},\hdots,\,
\sigma_{\omega(I_{j'})}^{\eta}\right),\,  1\leq j<j'\leq p
\end{eqnarray}
and
\begin{eqnarray}\nn
\dst \varepsilon_{_{\mb{X}_{I_j}}}(1)=
\ell _{_{\mb{X}_{I_j}}}\left(\sigma_{\alpha(I_j)}^{\eta},\hdots,\,
\sigma_{\omega(I_j)}^{\eta}\right),\,  1\leq j\leq p\,.
\end{eqnarray}
\nid Note that, when $\eta=1=\sigma_i$, $i=1,\hdots,\,d$,  $\varepsilon_{_{\mb{X}_{I_1},\hdots,\, \mb{X}_{I_p}}}(1,\hdots,\,1)$ coincides with the usual concept of extremal coefficient $\varepsilon_{_{\mb{X}}}$ of $\mb{X}$. Under this framework, the family of possible extremal coefficients of all sub-vectors of $\mb{X}$ is characterized in Strokorb and Schlather (\cite{strok+schla12} 2012).

Moreover, since $F_{_{\mb{X}_I}}$ is a multivariate extreme values (MEV) model, we have, for $\mb{t}=(t_1,\hdots,t_d)$,
\begin{eqnarray}\nn
\dst
\bigwedge_{j=1}^p\ell _{_{\mb{X}_{I_j}}}(\mb{t}_{_{I_j}})\leq \ell_{_{\mb{X}_I}}(\mb{t})\leq \sum_{j=1}^p\ell _{_{\mb{X}_{I_j}}}(\mb{t}_{_{I_j}}),
\end{eqnarray}
which, along with Proposition \ref{pro1}, alow us to bound the extremal dependence function of $\mb{X}_{I_1},\hdots,\mb{X}_{I_p}$.

\begin{pro}\label{pro2}
If $\mb{X}_I=(X_1,\hdots,\, X_d)$ satisfies conditions (i) and (ii) then, for all $(\lambda_1,\hdots,\, \lambda_p)\in (0,\infty)^p$, we have
$$
 \bigwedge_{j=1}^p\lambda_j^{-1}\varepsilon_{_{\mb{X}_{I_j}}}(1)\leq \varepsilon_{_{\mb{X}_{I_1},\hdots,\, \mb{X}_{I_p}}}(\lambda_1,\hdots,\,\lambda_p)\leq \sum_{j=1}^p\lambda_j^{-1}\varepsilon_{_{\mb{X}_{I_j}}}(1),
$$
with the upper bound corresponding to independent random vectors $\mb{X}_{I_1},\hdots,\, \mb{X}_{I_p}$ and the lower bound to totally dependent margins $X_1,\hdots, X_d$.

\end{pro}

\nid Observe that, if $\mb{X}_{I_1},\hdots,\, \mb{X}_{I_p}$ are totally dependent vectors, then the copula of $\mb{X}$ is the minimum copula (Nelsen \cite{nel06} 2006).\\

Now we analyze how  $\varepsilon_{_{\mb{X}_{I_i},\mb{X}_{I_{j'}}}}(\lambda_j,\lambda_{j'})$ relates with the dependence within the tails of $\mb{X}_{I_i}$ and $\mb{X}_{I_{j'}}$, $1\leq j<j'\leq p$.  Analogously to Ferreira and Ferreira (\cite{fer+fer12b} 2012), we are going to consider an upper tail dependence function of vector $(\mb{X}_{I_j},\mb{X}_{I_{j'}})$ given by the common value of
\begin{eqnarray}\label{chi00}
\dst \lim_{t\to\infty}P(M(I_j)>1-\lambda_j/t|M(I_{j'})>1-\lambda_{j'}/t)
\lambda_{j'}\varepsilon_{_{\mb{X}_{I_{j'}}}}(1)
\end{eqnarray}
and
\begin{eqnarray}\label{chi01}
\dst \lim_{t\to\infty}P(M(I_{j'})>1-\lambda_{j'}/t|M(I_{j})>1-\lambda_{j}/t)
\lambda_{j}\varepsilon_{_{\mb{X}_{I_{j}}}}(1).
\end{eqnarray}

Considering the first limit, observe that
\begin{eqnarray}\label{chi1}
\begin{array}{rl}
&\dst \lim_{t\to\infty}P(M(I_j)>1-\lambda_j/t|M(I_{j'})>1-\lambda_{j'}/t)
\\\\
=& \dst \lim_{t\to\infty}\left(1+ \frac{1-P(M(I_j)\leq 1-\lambda_j/t)}{1-P(M(I_{j'})\leq 1-\lambda_{j'}/t)}- \frac{1-P(M(I_j)\leq 1-\lambda_j/t,M(I_{j'})\leq 1-\lambda_{j'}/t)}{1-P(M(I_{j'})\leq 1-\lambda_{j'}/t)}
\right)
\end{array}
\end{eqnarray}
and that
\begin{eqnarray}\nn
\begin{array}{rl}
&\dst \lim_{t\to\infty}t\,P(M(I_j)\leq 1-\lambda_j/t,M(I_{j'})\leq 1-\lambda_{j'}/t)\\\\
=&-\ln C_{_{\mb{X}_{I_j}, \mb{X}_{I_{j'}}}}(e^{-\lambda_j},\hdots,e^{-\lambda_j},e^{-\lambda_{j'}},\hdots,e^{-\lambda_{j'}}),
\end{array}
\end{eqnarray}
since $C_{_{\mb{X}_{I_j}, \mb{X}_{I_{j'}}}}$ is max-stable. By Lemma \ref{lem1}, we obtain
\begin{eqnarray}\nn
\begin{array}{rl}
&-\ln C_{_{\mb{X}_{I_j}, \mb{X}_{I_{j'}}}}(e^{-\lambda_j},\hdots,e^{-\lambda_j},e^{-\lambda_{j'}},\hdots,e^{-\lambda_{j'}})\\\\
=& \ell_{_{\mb{X}_{I_j\cup I_{j'}}}}\left(\left(\frac{\sigma_{\alpha(I_j)}}{\lambda_j}\right)^{\eta},\hdots,
\left(\frac{\sigma_{\omega(I_j)}}{\lambda_j}\right)^{\eta},
\left(\frac{\sigma_{\alpha(I_{j'})}}{\lambda_{j'}}\right)^{\eta},\hdots,
\left(\frac{\sigma_{\omega(I_{j'})}}{\lambda_{j'}}\right)^{\eta}\right)\,.
\end{array}
\end{eqnarray}
By the homogeneity of order $-1/\eta$ of $\ell$, the limit in (\ref{chi1}) becomes
\begin{eqnarray}\nn
1+\frac{\lambda_j\varepsilon_{_{\mb{X}_{I_j}}}(1)}{\lambda_{j'}\varepsilon_{_{\mb{X}_{I_{j'}}}}(1)}
-\frac{\varepsilon_{_{\mb{X}_{I_j}, \mb{X}_{I_{j'}}}}(\lambda_j^{-1},\lambda_{j'}^{-1})}{\lambda_{j'}\varepsilon_{_{\mb{X}_{I_{j'}}}}(1)}
\end{eqnarray}
Switching the roles of $j$ and $j'$ in the conditional probabilities, we can see that both functions in (\ref{chi00}) and (\ref{chi01}) are equal and its common value is given in the following definition.

\begin{defi}
For $\mb{X}_I=(X_1,\hdots,X_d)$ under conditions (i) and (ii) and $1\leq j<j'\leq p$, the tail dependence function $\chi_{_{\mb{X}_{I_j},\mb{X}_{I_{j'}}}}(\lambda_j,\lambda_{j'})$ for $(\mb{X}_{I_j},\mb{X}_{I_{j'}})$ is defined by
\begin{eqnarray}\nn
\chi_{_{\mb{X}_{I_j},\mb{X}_{\mb{X}_{I_{j'}}}}}(\lambda_j,\lambda_{j'})
=\lambda_{j}\varepsilon_{_{\mb{X}_{I_{j}}}}(1)+\lambda_{j'}\varepsilon_{_{\mb{X}_{I_{j'}}}}(1)
-\varepsilon_{_{\mb{X}_{I_j}, \mb{X}_{I_{j'}}}}(\lambda_j^{-1},\lambda_{j'}^{-1})
\end{eqnarray}
and the value $\chi_{_{\mb{X}_{I_j},\mb{X}_{I_{j'}}}}(1,1)\equiv \chi_{_{\mb{X}_{I_j},\mb{X}_{I_{j'}}}}$ is denoted by coefficient of tail dependence for $(\mb{X}_{I_j},\mb{X}_{I_{j'}})$.\\
\end{defi}

In the following we present a property of the generalized madogram coming from the function $\varepsilon_{_{\mb{X}_{I_1},\hdots,\, \mb{X}_{I_p}}}(\lambda_1,\hdots,\,\lambda_p)$.
\begin{pro}\label{pro2}
If $\mb{X}_I=(X_1,\hdots,\, X_d)$ satisfies conditions (i) and (ii) then, for all $(\lambda_1,\hdots,\, \lambda_p)\in (0,\infty)^p$,
\begin{eqnarray}\nn
\dst \nu_{_{\mb{X}_{I_1},\hdots,\, \mb{X}_{I_p}}}(\lambda_1,\hdots,\,\lambda_p)=\frac{\varepsilon_{_{\mb{X}_{I_1},\hdots,\, \mb{X}_{I_p}}}(\lambda_1,\hdots,\,\lambda_p)}{1+\varepsilon_{_{\mb{X}_{I_1},\hdots,\, \mb{X}_{I_p}}}(\lambda_1,\hdots,\,\lambda_p)}
-\frac{1}{p}\sum_{j=1}^p\frac{\varepsilon_{_{\mb{X}_{I_j}}}(\lambda_j)}
{1+\varepsilon_{_{\mb{X}_{I_j}}}(\lambda_j)}.
\end{eqnarray}
\end{pro}

In particular, considering $p=d=2$ and $\lambda_1=\lambda_2=1$, we recover the initial relation between the madogram $\nu$ and the extremal coefficient $\varepsilon$, given by ${\nu=\frac{\varepsilon-1}{2(\varepsilon+1)}}$ (Cooley \ea~\cite{coo+06} 2006).

\section{Examples}\label{sex}

Consider $r\geq 1$ integer, $\beta_{ji}$, $i=1,\hdots,d$, $j=1,\hdots,r$, non negative constants such that $\sum_{j=1}^r\beta_{ji}=1$,  $i=1,\hdots,d$, and $\alpha_j$, $j=1,\hdots,r$, constants in $(0,1]$. Consider $C_j$, $j=1,\hdots,r$,  max-stable copulas and define
\begin{eqnarray}\label{coplog}
C_{\eta}(u_1,\hdots,u_d)=\exp\es\{-\sum_{j=1}^r\es(-\ln C_j\es(e^{-(-\beta_{j1}\ln u_1)^{\eta/\alpha_j}},\hdots,e^{-(-\beta_{jd}\ln u_d)^{\eta/\alpha_j}}\di) \di)^{\alpha_j/\eta}\di\}\,,
\end{eqnarray}
with $\eta\in(0,1]$ and such that $\alpha_j/\eta\in(0,1]$. This parametric family of copulas can be obtained from a mixture model of various MEV distributions (Ferreira and Pereira \cite{fer+per11} 2011) and encompasses several known copulas such as logistic symmetric and asymmetric and geometric means.

Consider $\mb{X}_I$ has marginals in (i) and copula in (\ref{coplog}). Then
\begin{eqnarray}\nn
F_{_{\mb{X}_I}}(t_1,\hdots,t_d)=\exp\es\{-\sum_{j=1}^r\es(-\ln C_j\es(e^{-(\beta_{j1}\sigma_1 t_1^{-1/\eta} )^{\eta/\alpha_j}},\hdots,e^{-(\beta_{jd}\sigma_d t_d^{-1/\eta})^{\eta/\alpha_j}}\di) \di)^{\alpha_j/\eta}\di\}\,.
\end{eqnarray}
The tail dependence function $\ell_{_{\mb{X}_I}}(t_1,\hdots,t_d)$ is homogeneous of order $-1/\eta$ and thus we are in the context of the previous section. We will consider different particular cases in the choice of the constants and MEV copulas and we determine the respective extremal coefficients and coefficients of tail dependence.

\begin{ex}\label{ex1}
Considering $r=1$, $\beta_{1i}=1$, $i=1,\hdots,d$, we obtain
\begin{eqnarray}\nn
F_{_{\mb{X}_I}}(t_1,\hdots,t_d)=\exp\es\{-\es(-\ln C\es(e^{-(\sigma_1 t_1^{-1/\eta} )^{\eta/\alpha}},\hdots,e^{-(\sigma_d t_d^{-1/\eta})^{\eta/\alpha}}\di) \di)^{\alpha/\eta}\di\}\,.
\end{eqnarray}
and if we take $C=\prod$, we find
\begin{eqnarray}\nn
\begin{array}{rl}
F_{_{\mb{X}_I}}(t_1,\hdots,t_d)=&\exp\es\{-\es((\sigma_1 t_1^{-1/\eta} )^{\eta/\alpha}+\hdots+(\sigma_d t_d^{-1/\eta})^{\eta/\alpha} \di)^{\alpha/\eta}\di\}\\\\
=& \exp\es\{-\es(\sigma_1^{\eta/\alpha} t_1^{-1/\alpha} +\hdots+\sigma_d^{\eta/\alpha} t_d^{-1/\alpha} \di)^{\alpha/\eta}\di\}\,.
\end{array}
\end{eqnarray}
We have
\begin{eqnarray}\nn
\begin{array}{ccc}
\varepsilon_{_{\mb{X}_{I_1},\hdots,\mb{X}_{I_p}}}(1,\hdots,1)=d^{\alpha/\eta}\,, &
\varepsilon_{_{\mb{X}_{I_j}}}(1)=|I_j|^{\alpha/\eta}\,,&
\varepsilon_{_{\mb{X}_{I_j},\mb{X}_{I_{j'}}}}(1,1)=|I_j\cup I_{j'}|^{\alpha/\eta}\,,
\end{array}
\end{eqnarray}
\begin{eqnarray}\nn
\begin{array}{c}
\dst \nu_{_{\mb{X}_{I_1},\hdots,\mb{X}_{I_p}}}(1,\hdots,1)=\frac{d^{\alpha/\eta}}{1+d^{\alpha/\eta}}
-\frac{1}{p}\sum_{j=1}^p\frac{|I_j|^{\alpha/\eta}}{1+|I_j|^{\alpha/\eta}}
\end{array}
\end{eqnarray}
and
\begin{eqnarray}\nn
\begin{array}{c}
\chi_{_{\mb{X}_{I_j},\mb{X}_{I_{j'}}}}=|I_j|^{\alpha/\eta}+|I_{j'}|^{\alpha/\eta}-
(|I_j|+|I_{j'}|)^{\alpha/\eta}\,,
\end{array}
\end{eqnarray}
for all $1\leq j<j'\leq d$, this latter generalizing the known result $\chi_{_{{X}_{j},{X}_{j'}}}=2-2^{\alpha/\eta}$ for the logistic model.
\end{ex}

\begin{ex}\label{ex2}
Considering the previous example with positive constants $\beta_{1i}=\beta_i$, $i=1,\hdots,d$, not necessarily equal to $1$, we have
\begin{eqnarray}\nn
F_{_{\mb{X}_I}}(t_1,\hdots,t_d)=&\exp\es\{-\es((\beta_1\sigma_1)^{\eta/\alpha} t_1^{-1/\alpha}+\hdots+(\beta_d\sigma_d)^{\eta/\alpha} t_d^{-1/\alpha} \di)^{\alpha/\eta}\di\}
\,.
\end{eqnarray}
We obtain
\begin{eqnarray}\nn
\begin{array}{c}
\dst\varepsilon_{_{\mb{X}_{I_1},\hdots,\mb{X}_{I_p}}}(1,\hdots,1)
=\es(\beta_1^{\eta/\alpha}+\hdots+\beta_d^{\eta/\alpha}\di)^{\alpha/\eta}\,, \\\\
\dst
\varepsilon_{_{\mb{X}_{I_j}}}(1)=\es(\sum_{i\in I_j}\beta_i^{\eta/\alpha}\di)^{\alpha/\eta}\,,\,\,\,\,\dst
\varepsilon_{_{\mb{X}_{I_j},\mb{X}_{I_{j'}}}}(1,1)=\es(\sum_{i\in I_j\cup I_{j'}}\beta_i^{\eta/\alpha}\di)^{\alpha/\eta}
\end{array}
\end{eqnarray}
and
\begin{eqnarray}\nn
\begin{array}{c}
\dst \chi_{_{\mb{X}_{I_j},\mb{X}_{I_{j'}}}}=\es(\sum_{i\in I_j}\beta_i^{\eta/\alpha}\di)^{\alpha/\eta}+\es(\sum_{i\in I_{j'}}\beta_i^{\eta/\alpha}\di)^{\alpha/\eta}-
\es(\sum_{i\in I_j\cup I_{j'}}\beta_i^{\eta/\alpha}\di)^{\alpha/\eta}\,,
\end{array}
\end{eqnarray}
for all $1\leq j<j'\leq d$.
\end{ex}

The previous examples consist in asymmetric logistic models. In the following we consider $\beta_{ji}=\beta_j$, $i=1,\hdots,d$, and $r>1$, i.e., weighted geometric means.

\begin{ex}\label{ex3}
Consider $r=2$, ${C_1=\bigwedge}$ and $C_2=\prod$. We have
\begin{eqnarray}\nn
\begin{array}{rl}
F_{_{\mb{X}_I}}(t_1,\hdots,t_d)=&\dst\prod_{j=1}^2\exp\es\{-\beta_j\es(-\ln C_j\es(e^{-\es(\sigma_1 t_1^{-1/\eta}\di)^{\eta/\alpha}} ,\hdots,e^{-\es(\sigma_d t_d^{-1/\eta}\di)^{\eta/\alpha}}\di)\di)^{\alpha/\eta}\di\}\\\\
=&\dst\exp\es\{-\beta_1\es(\bigvee_{i=1}^d\es(\sigma_i t_i^{-1/\eta}\di)^{\eta/\alpha}\di)^{\alpha/\eta}-(1-\beta_1)\es(\sum_{i=1}^d\es(\sigma_i t_i^{-1/\eta}\di)^{\eta/\alpha}\di)^{\alpha/\eta}\di\}\\\\
=&\dst\exp\es\{-\beta_1\bigvee_{i=1}^d\es(\sigma_i t_i^{-1/\eta}\di)-(1-\beta_1)\es(\sum_{i=1}^d\es(\sigma_i t_i^{-1/\eta}\di)^{\eta/\alpha}\di)^{\alpha/\eta}\di\}\,.
\end{array}
\end{eqnarray}
Thus we obtain
\begin{eqnarray}\nn
\begin{array}{c}
\dst\varepsilon_{_{\mb{X}_{I_1},\hdots,\mb{X}_{I_p}}}(1,\hdots,1)
=\beta_1+(1-\beta_1)d^{\alpha/\eta}=\beta_1\es(1-d^{\alpha/\eta}\di)+d^{\alpha/\eta}\,, \\\\
\dst
\varepsilon_{_{\mb{X}_{I_j}}}(1)=\beta_1+(1-\beta_1)|I_j|^{\alpha/\eta}
\end{array}
\end{eqnarray}
and
\begin{eqnarray}\nn
\begin{array}{rl}
\dst \chi_{_{\mb{X}_{I_j},\mb{X}_{I_{j'}}}}=&\dst \beta_1+(1-\beta_1)|I_j|^{\alpha/\eta}
+ \beta_1+(1-\beta_1)|I_{j'}|^{\alpha/\eta}-\beta_1-(1-\beta_1)|I_j\cup I_{j'}|^{\alpha/\eta}\\\\
=&\dst \beta_1\es(1-|I_{j}|^{\alpha/\eta}-|I_{j'}|^{\alpha/\eta}+(|I_{j}|+|I_{j'}|)^{\alpha/\eta}\di)
+|I_{j}|^{\alpha/\eta}+|I_{j'}|^{\alpha/\eta}-(|I_{j}|+|I_{j'}|)^{\alpha/\eta}\,,
\end{array}
\end{eqnarray}
for all $1\leq j<j'\leq d$.
\end{ex}

\section{A note on asymptotic tail independence}\label{sati}

In MEV models satisfying (i) and (ii), we only have tail dependence or tail independence between two marginals $X_j$ and $X_{j'}$ in the sense of
\begin{eqnarray}\nn
\chi_{_{{X}_{j},{X}_{j'}}}=\lim_{t\to\infty}P(F_j({X}_{j})>1-1/t,F_{j'}({X}_{j'})>1-1/t),
\end{eqnarray}
being, respectively, positive and null. Just observe that
\begin{eqnarray}\nn
\begin{array}{l}
\dst P(F_j({X}_{j})>1-1/t,F_{j'}({X}_{j'})>1-1/t)=2t^{-1}-1
+P\es({X}_{j}<\es(-\frac{\ln(1-1/t)}{\sigma_{j}}\di)^{-\eta},
{X}_{j'}<\es(-\frac{\ln(1-1/t)}{\sigma_{j'}}\di)^{-\eta}\di)\\\\
\sim  2t^{-1}-1
+P\es({X}_{j}<\es(\frac{t^{-1}}{\sigma_{j}}\di)^{-\eta},
{X}_{j'}<\es(\frac{t^{-1}}{\sigma_{j'}}\di)^{-\eta}\di)= 2t^{-1}-1+\exp\es\{-\ell_{_{({X}_{j},{X}_{j'})}}\es((t\sigma_{j})^{\eta},(t\sigma_{j'})^{\eta}\di)\di\}
\\\\
\sim  
2t^{-1}-t^{-1}\ell_{_{({X}_{j},{X}_{j'})}}\es(\sigma_{j}^{\eta},\sigma_{j'}^{\eta}\di)
+t^{-2}\frac{\es(\ell_{_{({X}_{j},{X}_{j'})}}\es(\sigma_{j}^{\eta},\sigma_{j'}^{\eta}\di)\di)^2}
{2}
\sim \left\{
\begin{array}{ll}
t^{-1}(2-\ell_{_{({X}_{j},{X}_{j'})}})&,\, \textrm{if }\ell_{_{({X}_{j},{X}_{j'})}}<2\\
t^{-2}\frac{\es(\ell_{_{({X}_{j},{X}_{j'})}}\es(\sigma_{j}^{\eta},\sigma_{j'}^{\eta}\di)\di)^2}{2}&,\, \textrm{if }\ell_{_{({X}_{j},{X}_{j'})}}=2,
\end{array}
\right.
\end{array}
\end{eqnarray}
the first branch corresponding to tail dependence ($\chi_{_{{X}_{j},{X}_{j'}}}=2-\ell_{_{({X}_{j},{X}_{j'})}}$) and the second to independence ($\chi_{_{{X}_{j},{X}_{j'}}}=0$).
However, non-negligible dependence may occur even when we have independence in the limit. A classical example in this context is the multivariate Gaussian model, whose bivariate marginals are asymptotic independent whatever the correlation parameters $\rho_{jj'}<1$. This phenomenon was also noticed in real data applications (see, e.g., Tawn (\cite{taw90} 1990), Guillou \ea ~\cite{gui+14} 2014 and references therein). Ledford and Tawn (\cite{led+taw96} 1996) addresses the modeling of the decay rate of the dependence under asymptotic independence. More precisely, they consider
\begin{eqnarray}\label{LTcond}
\begin{array}{rl}
\dst P(F_j({X}_{j})>1-1/t,F_{j'}({X}_{j'})>1-1/t)=t^{-1/\kappa_{_{{X}_{j},{X}_{j'}}}}\mathcal{L}(t),
\end{array}
\end{eqnarray}
where $\mathcal{L}$ is a slowly varying function (i.e., $\mathcal{L}(s)$, $s>0$, is a real function such that $\mathcal{L}(tx)/\mathcal{L}(t)\to 1$, as $t\to\infty$, $\forall x>0$) and $\kappa_{_{{X}_{j},{X}_{j'}}}\in (0,1]$ is denoted \emph{coefficient of asymptotic tail independence}. Observe that MEV sub-vectors  $(X_j,X_{j'})$ satisfy (\ref{LTcond}) with $\kappa_{_{{X}_{j},{X}_{j'}}}=1$ and $\mathcal{L}(t)=2-\ell_{_{({X}_{j},{X}_{j'})}}$ under tail dependence and $\kappa_{_{{X}_{j},{X}_{j'}}}=1/2$ and $\mathcal{L}(t)=2$ under independence.

In our context of MEV models, we also have
\begin{eqnarray}\nn
\dst \chi_{_{\mb{X}_{I_j},\mb{X}_{I_{j'}}}}=\lim_{t\to\infty}P(M(I_j)>1-1/t,M(I_{j'})>1-1/t)>0\,,
\end{eqnarray}
unless the marginals are independent. If we move to a broader framework than the MEV models, by a similar reasoning as in Ledford and Tawn (\cite{led+taw96} 1996),  2012), we assume
\begin{eqnarray}\label{LTcondMs}
\dst P(M(I_j)>1-1/t,M(I_{j'})>1-1/t)
=t^{-1/\kappa_{_{\mb{X}_{I_j},\mb{X}_{I_{j'}}}}}\mathcal{L}_{_{\mb{X}_{I_j},\mb{X}_{I_{j'}}}}(t)\,,
\end{eqnarray}
where function $\mathcal{L}_{_{\mb{X}_{I_j},\mb{X}_{I_{j'}}}}$ is slowly varying and $\kappa_{_{\mb{X}_{I_j},\mb{X}_{I_{j'}}}}\in (0,1]$  corresponds to the block coefficient of asymptotic tail independence introduced in Ferreira and Ferreira (\cite{fer+fer12b}). Under the validity of condition
\begin{eqnarray}\label{condetaMs}
\dst P(\min_{j\in S}\{F_j(X_j)\}>1-1/t,\min_{j'\in T}\{F_{j'}(X_{j'})\}>1-1/t)
=t^{-1/\kappa_{_{\mb{X}_{S},\mb{X}_{T}}}}\mathcal{L}_{_{\mb{X}_{S},\mb{X}_{T}}}(t)\,,
\end{eqnarray}
for all $\emptyset\not= S\subset I_j$ and $\emptyset\not= T\subset I_{j'}$, where the respective functions $\mathcal{L}_{_{\mb{X}_{S},\mb{X}_{T}}}$ are slowly varying, we can relate $\kappa_{_{\mb{X}_{I_j},\mb{X}_{I_{j'}}}}$ with the bivariate $\kappa_{_{{X}_{j},{X}_{j'}}}$, for $j\in I_j$ and $j'\in I_{j'}$. More precisely, by Proposition 2.9 in Ferreira and Ferreira (\cite{fer+fer12b} 2012), we have
\begin{eqnarray}\nn
\dst \kappa_{_{\mb{X}_{I_j},\mb{X}_{I_{j'}}}}=\max\{\kappa_{_{{X}_{j},{X}_{j'}}}:\, j\in I_j,\, j'\in I_{j'}\}\,.
\end{eqnarray}

Consider $\mb{X}_I=(X_1,\hdots,X_d)$ has an inverted MEV copula, that is, the survival copula $\overline{C}_{_{\mb{X}_I}}(u_1,\hdots, u_d)=P(F_1(X_1)\geq u_1,\hdots,F_d(X_d)\geq u_d)$ is expressed by
\begin{eqnarray}\nn
\dst \overline{C}_{_{\mb{X}_I}}(u_1,\hdots, u_d)=\exp\es\{-\ell_{_{\mb{Y}_I}}(-1/\ln(1-u_1),\hdots,-1/\ln(1-u_d))\di\},
\end{eqnarray}
where $\ell_{_{\mb{Y}_I}}$ is an exponent measure of some MEV distributed $\mb{Y}_I=(Y_1,\hdots,Y_d)$ (Wadsworth and Tawn \cite{wad+taw12} 2012). Assuming that $\mb{Y}_I$ satisfies conditions (i) and (ii), we have
\begin{eqnarray}\nn
\begin{array}{rl}
\dst &P(F_j({X}_{j})>1-1/t,F_{j'}({X}_{j'})>1-1/t)=\exp\es\{-\ell_{({Y}_{j},{Y}_{j'})}\es(\es(-\frac{\sigma_i}{\ln (1/t)}\di)^{\eta},\es(-\frac{\sigma_j}{\ln (1/t)}\di)^{\eta} \di)\di\}\\\\
=& \exp\es\{-(-\ln (1/t))\ell_{({Y}_{j},{Y}_{j'})}\es(\sigma_i^{\eta},\sigma_j^{\eta}\di)\di\}
=t^{-\ell_{({Y}_{j},{Y}_{j'})}\es(\sigma_i^{\eta},\sigma_j^{\eta}\di)},
\end{array}
\end{eqnarray}
and thus $\kappa_{_{{X}_{j},{X}_{j'}}}=1/\ell_{({Y}_{j},{Y}_{j'})}\es(\sigma_i^{\eta},\sigma_j^{\eta}\di)$.
Moreover, it is straightforward that, for any $A\subseteq I$,
\begin{eqnarray}\nn
\begin{array}{rl}
\dst &P(\min_{j\in \mb{X}_{A}}\{F_j(X_j)\}>1-1/t)=
\exp\es\{-\ell_{_{\mb{Y}_{A}}}\es(\es(\frac{\sigma_{\alpha(A)}}{\ln(1/t)}\right)^{\eta},\hdots,
\left(\frac{\sigma_{\omega(A)}}{\ln (1/t)}\right)^{\eta}\di)\di\}\\\\
=& \exp\es\{-(-\ln (1/t))\ell_{_{\mb{Y}_{A}}}\es(\sigma_{\alpha(A)}^{\eta},\hdots,
\sigma_{\omega(A)}^{\eta}\di)\di\}
=t^{-\ell_{_{\mb{Y}_{A}}}\es(\sigma_{\alpha(A)}^{\eta},\hdots,
\sigma_{\omega(A)}^{\eta}\di)},
\end{array}
\end{eqnarray}
and so (\ref{condetaMs}) holds with $\kappa_{_{\mb{X}_{A}}}=1/\ell_{_{\mb{Y}_{A}}}\es(\sigma_{\alpha(A)}^{\eta},\hdots,
\sigma_{\omega(A)}^{\eta}\di)$. Therefore, by Proposition 2.9 in Ferreira and Ferreira (\cite{fer+fer12b} 2012), we have
\begin{eqnarray}\nn
\dst \kappa_{_{\mb{X}_{I_j},\mb{X}_{I_{j'}}}}=1/
\min\{\ell_{({Y}_{j},{Y}_{j'})}\es(\sigma_i^{\eta},\sigma_j^{\eta}\di):\, j\in I_j,\, j'\in I_{j'}\}\,.
\end{eqnarray}

Models for $\mb{X}_I=(X_1,\hdots,X_d)$ satisfying (\ref{LTcondMs}) can be derived from Section \ref{sex}, by considering in Examples \ref{ex1}-\ref{ex3} that $(F_1(X_1),\hdots,F_d(X_d))$ has survival copula $\overline{C}(u_1,\hdots,u_d)={C}_{\eta}(1-u_1,\hdots,1-u_d)$, with ${C}_{\eta}$ given in (\ref{coplog}).\\

In a future work we will apply the models and measures here developed in real data, by following a similar approach to that of Guillou \emph{et al.} (\cite{gui+14} 2014). More precisely, since $P(\max(X_1,\hdots,X_d)\leq t)=\exp(-\ell_{_{\mb{X}_I}}(\mb{1}_I)t^{-1/\eta})$, $\eta$ can be estimated as the tail index of an extreme value model, like the Generalized Probability Weighted Moment approach (Diebolt \emph{et al.} \cite{die+08} 2008) or use the maximum likelihood (ML) estimator. Condition (i) also allows to derive ML estimators for $\sigma_i$, $i=1,\hdots,d$, where $\eta$ can be replaced by the ML estimate. Based on $P(\bigcap_{i\in I_j}X_i/\sigma_i^{\eta}\leq t)=\exp(-\varepsilon_{_{\mb{X}_{I_j}}}(1)\,t^{-1/\eta})$, an ML estimator for $ \varepsilon_{_{\mb{X}_{I_j}}}(1)$ can be deduced, with $\sigma_i$ and $\eta$ replaced by the respective ML estimates. Similarly we obtain ML estimators for ${\varepsilon_{_{\mb{X}_{I_j}, \mb{X}_{I_{j'}}}}}(1,1)$ and $\varepsilon_{_{\mb{X}_{I_1},\hdots,\, \mb{X}_{I_p}}}(1,\hdots,1)$.

Relation (\ref{lem2.1}) also leads us to alternative estimators for $\varepsilon_{_{\mb{X}_{I_1},\hdots,\, \mb{X}_{I_p}}}(1,\hdots,1)$, ${\varepsilon_{_{\mb{X}_{I_j}, \mb{X}_{I_{j'}}}}}(1,1)$ and $ \varepsilon_{_{\mb{X}_{I_j}}}(1)$. This approach is developed in Ferreira and Ferreira (\cite{fer+fer12b} 2012). See also Fonseca \emph{et al.} (\cite{fon+15} 2015). More precisely, we can state
\begin{eqnarray}\nn
\widehat{\varepsilon}_{_{\mb{X}_{I_1},\hdots,\, \mb{X}_{I_p}}}(1,\hdots,1)
=\frac{1}{1-\overline{\bigvee_{j=1}^{p}\bigvee_{i\in I_j}\widehat{F}_i(X_i)}}-1,
\end{eqnarray}
where $\widehat{F}_i$ is an estimator of the marginal df $F_i$, e.g., the empirical df  and notation $\overline{W}$ corresponds to the sample mean based on independent copies $W^{(l)}$, $l=1,\hdots,n$, of $W$. Analogously, we derive estimators ${\widehat{\varepsilon}_{_{\mb{X}_{I_j}, \mb{X}_{I_{j'}}}}}(1,1)$ and $ \widehat{\varepsilon}_{_{\mb{X}_{I_j}}}(1)$. Asymptotic properties are addressed in the given references.

%


\begin{thebibliography}{000}

\bibitem{beirl+04} Beirlant, J., Goegebeur, Y., Segers, J. and Teugels, J.L. (2004).
Statistics of Extremes: Theory and Applications. John Wiley \& Sons.

\bibitem{coo+06} Cooley, D., Naveau, P. and Poncet, P. (2006). Variograms for spatial maxstable random fields. In: Dependence in probability and statistics, Lecture Notes
in Statistics, 187, 373-390, Springer, New-York.

\bibitem{die+08} Diebolt, J., Guillou, A., Naveau, P. and Ribereau, P. (2008). Improving
probability-weighted moment methods for the generalized extreme value distribution, REVSTAT Statistical Journal 6, 35-50.

\bibitem{fer+fer12b} Ferreira, H. and Ferreira, M. (2012). On extremal dependence of block vectors. {Kybernetika} {48(5)}, 988-1006.

\bibitem{fer+per11} Ferreira,  H.  and  Pereira,  L.  (2011). Generalized  logistic  models  and  its  orthant  tail  dependence. Kybernetica 47(5), 732-739.

\bibitem{fer+per15} Ferreira, H. and Pereira, L. Dependence of maxima in space. IC-MSQUARE 2014 IOP Publishing, Journal of Physics: Conference Series 574 (2015) 012021.

\bibitem{fin+roo03} Finkenst\"{a}dt, B. and Rootzén, H. (2003). Extreme values in finance, telecommunication and the environment. Chapman \& Hall.

\bibitem{fon+15} Fonseca, C., Pereira, L., Ferreira, H. and Martins, A.P. (2015). Generalized madogram and pairwise dependence of maxima over two regions of a random field. Kybernetika 51(2), 193-211.

\bibitem{gui+14} Guillou, A., Naveau, P., and Schorgen, A. (2014). Madogram and asymptotic independence among maxima. REVSTAT - Statistical Journal
12(2), 119-134.

\bibitem{led+taw96} Ledford, A., Tawn, J.A. (1996). Statistics for near independence in multivariate extreme values. Biometrika 83, 169--187.

\bibitem{marc+16} Marcon, G., Padoan, S.A., Naveau, P., Muliere,  P. and Segers, J. (2016). Multivariate nonparametric estimation of the Pickands dependence function using Bernstein polynomials. Journal of Statistical Planning and Inference. In press.

\bibitem{nav+09} Naveau, P., Guillou, A., Cooley, D. and Diebolt, J. (2009). Modelling
pairwise dependence of maxima in space. Biometrika, 96, 1-17.

\bibitem{nel06} Nelsen, R.B. (2006). An Introduction to Copulas. Second edition. Springer, New York.

\bibitem{ram+led11} Ramos, A. and Ledford, A. (2011). Alternative point process framework for modeling multivariate extreme values. Communications in Statistics - Theory and Methods, 40, 2205-2224.

\bibitem{strok+schla12} Strokorb, K. and Schlather, M. (2012). Characterizing extremal coefficient functions and extremal correlation functions. arXiv:1205.1315v1

\bibitem{taw90} Tawn, J.A. (1990). Discussion of paper by A. C. Davison and R. L. Smith. J. R. Statist. Soc. B 52, 248-9.

\bibitem{wad+taw12} Wadsworth, J.L. and Tawn, J.A. (2012). Dependence modelling for spatial extremes. Biometrika 99(2), 253-272.


\end{thebibliography}
\end{document}